\documentclass[]{article}

\usepackage{dsfont}
\usepackage{graphicx}
\usepackage{amsmath}
\usepackage{amssymb}
\usepackage{amsthm}
\usepackage{mathrsfs}
\DeclareMathAlphabet{\mathsc}{OT1}{cmr}{m}{sc}
\usepackage{pxfonts}
\usepackage{mathdots}
\usepackage{sectsty}
\usepackage{tikz}
\usepackage{adjustbox}
\usepackage{enumitem}
\usepackage{caption}
\usepackage{bbold}
\allowdisplaybreaks
\usepackage[authoryear,round]{natbib}
\usepackage[dvipsnames]{xcolor}
\usepackage[
  colorlinks=true,
  linkcolor=OliveGreen,
  citecolor=OliveGreen,
  urlcolor=OliveGreen,
  pdftitle={Sine beta is tail trivial},
  pdfauthor={Theodoros Assiotis and Kohei Suzuki}
]{hyperref}
\usepackage[nameinlink,capitalise,noabbrev]{cleveref}

\sectionfont{\scshape\centering\fontsize{11}{14}\selectfont}
\subsectionfont{\scshape\fontsize{11}{14}\selectfont}
\usepackage{fancyhdr}
\usepackage[nottoc,notlot,notlof]{tocbibind}

\newcommand\shorttitle{Collision and non-collision on configuration space}
\newcommand\authors{T. Assiotis and K. Suzuki}

\newtheorem{thm}{Theorem}[section]
\newtheorem{lemma}[thm]{Lemma}
\newtheorem{proposition}[thm]{Proposition}

\newtheorem{cor}[thm]{Corollary}
\newtheorem*{theorem*}{Theorem}

\newcommand{\E}{\mathcal E}
\newcommand{\D}{\mathcal D}

\newcommand{\supp}{\mathsf{supp}}

\newcommand{\R}{\mathbb{R}}
\newcommand{\N}{\mathbb{N}}

\newtheorem{theorem}[thm]{Theorem}
\crefname{thm}{Theorem}{Theorems}
\crefname{theorem}{Theorem}{Theorems}
\crefname{lemma}{Lemma}{Lemmas}
\crefname{proposition}{Proposition}{Propositions}
\crefname{cor}{Corollary}{Corollaries}
\crefname{remark}{Remark}{Remarks}
\crefname{definition}{Definition}{Definitions}

\newcommand{\Lip}{\mathsf{Lip}}
\newcommand{\dd}{\,\mathrm{d}}

\newcommand{\cF}{\mathscr{F}}
\newcommand{\cT}{\mathscr{T}}
\newcommand{\muB}{\boldsymbol{\mu}_\beta}
\newcommand{\angles}[1]{\left\langle #1\right\rangle}
\newcommand{\Span}{\mathop{\mathsf{span}}}

\newcommand{\U}{\mathsf{Conf}}

\newcommand{\cquad}{\ , \quad}
\newcommand{\comma}{\ ,}
\newcommand{\fstop}{\ .}

\newcommand{\cG}{\mathcal{G}}
\newcommand{\defi}{\overset{\mathrm{def}}{=}}
\renewcommand\shorttitle{$\mathsf{Sine}_\beta$ tail trivial}

\renewcommand\authors{T. Assiotis and K. Suzuki}

\title{\large\bfseries $\mathsf{Sine}_\beta$ is tail trivial}
\author{\small THEODOROS ASSIOTIS AND KOHEI SUZUKI}
\date{}

\begin{document}

\maketitle

\begin{abstract}
We prove that the \(\mathsf{Sine}_\beta\) point process is tail trivial for every \(\beta>0\). The key ingredient in the proof is a certain uniform entropy bound, that makes  use of the DLR equations of~\citet{DLR} and the bounded two-point correlation density obtained by~\citet{qu2025pair} and~\citet{AssiotisNajnudel}, that allows us to show that tail events are translation invariant. Tail triviality then follows either from the uniqueness theorem of~\citet{AssiotisDLRUniqueness} or from translation mixing, which we derive from the correlation estimates of~\citet{dumaz2026longrangecorrelationsinebetapoint}. As consequences, we obtain the following applications: (a)~extremality of $\muB$ in the full convex class of fixed-kernel Gibbs measures without stationarity and finite-energy conditions; (b)~$\mathrm{L}^2$-ergodicity of the associated infinite-particle reversible diffusion for every $\beta>0$; and (c) the almost-sure collision/noncollision dichotomy with threshold $\beta=1$ for the diffusion.
\end{abstract}

\section{The main result}
\paragraph{The $\mathsf{Sine}_{\beta}$ point process.}
Fix $\beta>0$. For each integer $N\geq1$, let $(\theta_1,\ldots,\theta_N)\in(-\pi,\pi]^N$ have
the circular $\beta$-ensemble density
$$
  \frac{1}{Z_{N,\beta}}
  \prod_{1\leq i<j\leq N}
  \left|\mathrm e^{\mathrm i\theta_i}-\mathrm e^{\mathrm i\theta_j}\right|^\beta
  \prod_{i=1}^N\frac{\dd\theta_i}{2\pi}.
$$
Since the mean angular spacing is $2\pi/N$, consider the unfolded
point process
$$
  \boldsymbol{\Xi}_N\defi\sum_{i=1}^N\delta_{N\theta_i/(2\pi)}.
$$
As $N\to\infty$, $\boldsymbol{\Xi}_N$ converges locally in distribution to a
translation-invariant simple point process $\boldsymbol{\Xi}_\beta$ on
$\mathbb R$ of intensity one. We denote its law by $\muB$ and call the process
$\mathsf{Sine}_{\beta}$. For the classical values $\beta=1,2,4$, the $\mathsf{Sine}_{\beta}$
processes were already known through their explicit integrable structures:
$\mathsf{Sine}_2$ is determinantal, whereas $\mathsf{Sine}_1$ and
$\mathsf{Sine}_4$ are Pfaffian point processes.  For general $\beta>0$,
the infinite-volume bulk process was constructed independently by~\citet{KillipStoiciu}, as the scaling limit of the circular $\beta$-ensemble,
and by~\citet{ValVir09}, as the bulk limit of the Gaussian $\beta$-ensemble
 via the Brownian-carousel description.  The two limiting processes were
 subsequently shown by~\citet{ValkoViragCircular} to coincide.
 It may be viewed as an infinite-volume equilibrium state of the
 one-dimensional logarithmic gas through the Dobrushin--Lanford--Ruelle (DLR)~formalism~\citep{DLR}, and as the reversible measure of an
 associated infinite-particle diffusion~\citep{Osa12, Osa13, Tsa16, Suz22b}. For $\beta=1,2,4$, this diffusion is identified below with infinite-dimensional Dyson Brownian motion.

\paragraph{Main result} Write $\N=\{1,2,\ldots\}$, $\N_0=\N\cup\{0\}$. The law~\(\muB\) is a Borel probability measure on the configuration space of locally finite counting measures, with multiplicities allowed,
$$\U(\R)\defi \biggl\{\boldsymbol{\gamma} =\sum_{i=1}^N \delta_{x_i}: N \in \N_0 \cup \{+\infty\},\ x_i \in \R,\ \boldsymbol{\gamma}(K)<\infty \ \text{for all compact $K \subset \R$}\biggr\},$$
endowed with the vague topology; $\mathscr B(\mathsf{Conf}(\R))$ denotes its Borel $\sigma$-field. Let
\[
 \mathsf{Conf}_{\mathrm{s}}(\R)
 \defi \{\boldsymbol{\gamma}\in\mathsf{Conf}(\R):
 \boldsymbol{\gamma}(\{x\})\leq 1\text{ for every }x\in\R\}
\]
be its simple subspace. The process is simple, so \(\muB(\mathsf{Conf}_{\mathrm{s}}(\R))=1\).
For a Borel set $D\subset\R$, let $\mathsf{Conf}(D)$ denote the configurations in $\mathsf{Conf}(\R)$ supported in $D$. Write $\boldsymbol{\gamma}_D=\boldsymbol{\gamma}|_D$, and let
\[
  \cF_{D}\defi\sigma\bigl(\boldsymbol{\gamma} \mapsto \boldsymbol{\gamma}(B): B\subset D\text{ bounded Borel}\bigr).
\]
The {\it tail \(\sigma\)-field} on $\U(\R)$ is defined as
\[
  \cT\defi\bigcap_{R>0}\cF_{[-R,R]^c}.
\]
\begin{theorem}\label{thm:main}
$\mathsf{Sine}_\beta$ is tail trivial for every $\beta>0$. Namely, 
\[
  \muB(A)\in\{0,1\} \qquad A \in\cT.
\]
\end{theorem}
The tail triviality was previously known in the determinantal case $\beta=2$; see~\citep{BufetovQiuShamov2021, Lyons2018Tail, OsadaOsada2018}. Our result extends it to arbitrary $\beta>0$. 
\subsection{Applications}
\paragraph{Extremality}
An immediate application of Theorem~\ref{thm:main} is the extremality of $\muB$ as a Gibbs measure. 
Fix a bounded Borel set
$\Lambda\subset\mathbb{R}$ and put
$\boldsymbol{\eta}=\boldsymbol{\gamma}_{\Lambda^c}$, the restriction of $\boldsymbol{\gamma}$ to~$\Lambda^c=\R \setminus \Lambda$.
Number rigidity~\citep{ChhaibiNajnudel2018, DLR} yields a measurable function~$
N_\Lambda:\U(\Lambda^c)\longrightarrow\mathbb{N}_0$
such that
\[
\boldsymbol{\gamma}(\Lambda)
=
N_\Lambda(\boldsymbol{\gamma}_{\Lambda^c})
\qquad
\muB\text{-a.e.}
\]
For an exterior configuration $\boldsymbol{\eta}$, write~$N=N_\Lambda(\boldsymbol{\eta})$.
For $x\in\Lambda$, define the renormalized exterior weight by
\[
\omega(x\mid\boldsymbol{\eta})
\defi
\lim_{m\to\infty}
\prod_{\substack{u\in\boldsymbol{\eta}\\ |u|\leq m}}
\left|1-\frac{x}{u}\right|^\beta,
\]
where $m\in\N$ and, for $\muB$-a.e.~configuration~$\boldsymbol{\gamma}$, the limit exists with $\boldsymbol{\eta}=\boldsymbol{\gamma}_{\Lambda^c}$ (see~\citep{DLR}).
The corresponding partition function is
\[
Z_\Lambda(\boldsymbol{\eta})
\defi
\int_{\Lambda^N}
\prod_{1\leq i<j\leq N}|x_i-x_j|^\beta
 \prod_{i=1}^N\omega(x_i\mid\boldsymbol{\eta})
 \dd x_1\cdots \dd x_N.
\]
On a measurable set of exterior configurations of full exterior marginal measure, the DLR theorem gives both the existence of the preceding limit and
\(0<Z_\Lambda(\boldsymbol{\eta})<\infty\). For a bounded measurable function
$F:\U(\mathbb{R})\to\mathbb{R}$ and an exterior configuration in this set, define
\[
\begin{aligned}
(\Gamma^\beta_\Lambda F)(\boldsymbol{\eta})
\defi{}&
\frac{1}{Z_\Lambda(\boldsymbol{\eta})}
\int_{\Lambda^N}
F\left(\boldsymbol{\eta}+\sum_{i=1}^N\delta_{x_i}\right) 
\prod_{1\leq i<j\leq N}|x_i-x_j|^\beta
 \prod_{i=1}^N\omega(x_i\mid\boldsymbol{\eta})
 \dd x_1\cdots \dd x_N.
\end{aligned}
\]
On the null complement, set \(\Gamma^\beta_\Lambda(\boldsymbol{\eta},\cdot)=\delta_{\boldsymbol{\eta}}\); this gives an everywhere-defined measurable probability kernel. For a full configuration, use the convention
\[
 (\Gamma^\beta_\Lambda F)(\boldsymbol{\gamma})
 \defi (\Gamma^\beta_\Lambda F)(\boldsymbol{\gamma}_{\Lambda^c}).
\]
The DLR equation reads
\[
\mathbb{E}_{\muB}
\left[
F\mid\cF_{\Lambda^c}
\right](\boldsymbol{\gamma})
=
(\Gamma^\beta_\Lambda F)(\boldsymbol{\gamma}_{\Lambda^c})
\qquad
\muB\text{-a.e.}
\]
For a Borel probability measure $\nu$ on $\mathsf{Conf}(\R)$, we use the convention
\[
 \nu(F)=\mathbb E_\nu[F]\defi
 \int_{\mathsf{Conf}(\R)}F(\boldsymbol{\gamma})\,\nu(\dd\boldsymbol{\gamma}).
\]
For $\nu\in\mathcal{P}(\U(\mathbb{R}))$, we write~$\nu\Gamma^\beta_\Lambda=\nu$ if for every bounded Borel function~$F: \U(\R) \to \R$,
$$\nu(\Gamma^\beta_\Lambda F) =\nu(F).$$
Define the following convex subset of the space~$\mathcal P(\U(\R))$ of Borel probability measures on~$\U(\R)$:
\[
\cG(\Gamma^\beta)
\defi
\left\{
\nu\in\mathcal{P}(\U(\mathbb{R})):
\nu\Gamma^\beta_\Lambda=\nu
\text{ for every bounded Borel }\Lambda
\right\}
\]
and write $\operatorname{ex}\cG(\Gamma^\beta)$ for the set of extreme points of $\cG(\Gamma^\beta)$. By~\citet{DLR}, $\muB \in \cG(\Gamma^\beta)$. 

We stress that no stationarity or finite-energy assumption is imposed on the measures in $\cG(\Gamma^\beta)$. The following corollary asserts extremality in this full convex class, which is not a straightforward implication of~the uniqueness theorem of~\citet[Theorem~1.1]{AssiotisDLRUniqueness} because the uniqueness theorem concerns only the restricted class of stationary finite-energy solutions of the canonical DLR equations.   The direct proof from tail triviality that we give below rules out decompositions even into nonstationary fixed-kernel solutions.
\begin{cor}\label{cor:extremality} For every $\beta>0$, 
$$\muB\in\operatorname{ex}\cG(\Gamma^\beta).$$
\end{cor}
\begin{proof}
We use the usual implication from tail triviality to extremality; see~\citet[Theorem~7.7]{Geo11}. We include its short density argument for the fixed kernels defined above.
Suppose that $\muB=\alpha\nu_1+(1-\alpha)\nu_2$, where $0<\alpha<1$ and $\nu_1,\nu_2\in\cG(\Gamma^\beta)$. For each $i\in\{1,2\}$, we have $\nu_i\ll\muB$; set $f_i=\dd\nu_i/\dd\muB$. For every bounded Borel $\Lambda$ and every bounded Borel $F$, invariance under $\Gamma^\beta_\Lambda$ and the DLR equation give
\[
 \int_{\mathsf{Conf}(\R)}Ff_i\,\dd\muB
 =\int_{\mathsf{Conf}(\R)}(\Gamma^\beta_\Lambda F)f_i\,\dd\muB
 =\int_{\mathsf{Conf}(\R)}F\,\mathbb E_{\muB}[f_i\mid\cF_{\Lambda^c}]\,\dd\muB.
\]
Thus $f_i=\mathbb E_{\muB}[f_i\mid\cF_{\Lambda^c}]$ $\muB$-a.e. Taking $\Lambda=[-n,n]$, $n\in\N$, and applying the reverse martingale theorem shows that $f_i$ is measurable with respect to the completion of $\cT$. Theorem~\ref{thm:main} therefore makes $f_i$ constant; since $\int_{\mathsf{Conf}(\R)}f_i\,\dd\muB=1$, we have $f_i=1$ $\muB$-a.e. and hence $\nu_i=\muB$.
\end{proof}

\paragraph{Ergodicity}
Another application is the $\mathrm{L}^2$-ergodicity of the associated reversible diffusion process. 
Let $\mathcal D^{\mathrm{loc}}_\infty$ be the class of bounded local functions whose symmetric finite-volume representatives on every particle sector are $C^\infty$, as in~\cite{Osa96} (or ~\citet[Section~1.1]{assiotis2026collision}), and set
\[
 \mathcal E_\beta(F)
 \defi \frac12\int_{\U(\R)} \sum_{x\in \boldsymbol{\gamma}}
 |\nabla_x F(\boldsymbol{\gamma})|^2\,\muB(\dd\boldsymbol{\gamma})
 \in[0,\infty],\qquad F\in\mathcal D^{\mathrm{loc}}_\infty.
\]
Put
\[
 \mathcal D_{\beta,\circ}
 \defi \{F\in\mathcal D^{\mathrm{loc}}_\infty\cap \mathrm{L}^2(\muB):
 \mathcal E_\beta(F)<\infty\},
\]
and define the associated bilinear form on this pre-domain by
\[
 \mathcal E_\beta(F,G)
 \defi \frac14\bigl(\E_\beta(F+G)-\E_\beta(F-G)\bigr),
 \qquad F,G\in\mathcal D_{\beta,\circ};
\]
see also~\cite[(1.2)]{Osa96}. For example, if $m\in\N$, $f_1,\ldots,f_m\in C_c^\infty(\R)$ and $\Theta\in C_b^\infty(\R^m)$, then, for the cylinder function
\[
 F(\boldsymbol{\gamma})
 \defi\Theta\bigl(\angles{f_1,\boldsymbol{\gamma}},\ldots,
 \angles{f_m,\boldsymbol{\gamma}}\bigr),
\]
the particle gradient is the derivative obtained by moving the particle at $x$:
\[
 \nabla_xF(\boldsymbol{\gamma})
 =\left.\frac{\mathrm d}{\mathrm d y}
 F(\boldsymbol{\gamma}-\delta_x+\delta_y)\right|_{y=x}
 =\sum_{i=1}^m
 (\partial_i\Theta)\bigl(\angles{f_1,\boldsymbol{\gamma}},\ldots,
 \angles{f_m,\boldsymbol{\gamma}}\bigr)f_i'(x).
\]
For every $\beta>0$, $(\E_\beta,\mathcal D_{\beta,\circ})$ is closable on $\mathrm{L}^2(\muB)$, and its closure $(\E_\beta, \mathcal D_\beta)$ is a quasi-regular strongly local $\muB$-symmetric Dirichlet form, see~\cite{Osa96, Osa13} for $\beta=1,2,4$ and \cite{Suz22b} and \cite[Theorem~1.5]{assiotis2026collision} for $\beta>0$. Here and below, ``q.e.'' means outside a set of zero $(\E_\beta,\mathcal D_\beta)$-capacity.

By a fundamental theorem of Dirichlet form theory (e.g.,~\citet[Theorem~1.5.2]{CheFuk11}), $(\mathcal E_\beta, \mathcal D_\beta)$ therefore admits an associated $\muB$-reversible diffusion process in $\U(\R)$. Its transition semigroup on $\mathrm{L}^2(\U(\R),\muB)$ satisfies
\[
 T_t^{(\beta)}=\mathrm e^{t\mathcal A_\beta},
\]
where $(\mathcal A_\beta, \mathcal D(\mathcal A_\beta))$ is the infinitesimal generator, i.e., the unique self-adjoint operator satisfying
\[
 \E_\beta(F,G)=-\langle \mathcal A_\beta F,G\rangle_{\mathrm{L}^2(\muB)},
 \qquad F\in\mathcal D(\mathcal A_\beta),\quad G\in\D_\beta.
\]
Equivalently, for $f\in \mathrm{L}^2(\muB)$, $u_t=T_t^{(\beta)}f$ is the unique mild solution of
\[
 \partial_tu_t=\mathcal A_\beta u_t,\qquad u_0=f.
\]
Let $(\boldsymbol{\Xi}^{(\beta)}_t)_{t\geq0}$ denote the associated diffusion process in~$\U(\R)$, with law $\mathbb P^{(\beta)}_{\boldsymbol{\gamma}}$ when started from $\boldsymbol{\gamma}$; this law is defined for q.e.~$\boldsymbol{\gamma}\in\U(\R)$.
For $\beta=1,2,4$, choose a labelled version and write
$$  \boldsymbol{\Xi}^{(\beta)}_t=\sum_{i=1}^{\infty}\delta_{X_t^i} \fstop
$$
The labelled process satisfies the infinite-dimensional stochastic differential equation
$$
  \dd X_t^i
  =
  \dd B_t^i
  +
  \frac\beta2\lim_{r\to\infty}
  \sum_{\substack{j\in\N,\ j\neq i\\ |X_t^i-X_t^j|<r}}
  \frac{1}{X_t^i-X_t^j} \dd t
  \cquad i\in\mathbb N \comma
$$
for $\muB$-a.e.~initial configuration, where $\{B^i\}_{i\in\mathbb N}$ are independent standard Brownian
motions in $\mathbb R$. 
The limits are
defined through the logarithmic derivative of $\muB$ and
converge locally in $\mathrm{L}^p$ with respect to the $1$-Campbell measure for
every $1<p<2$, see~\citet[Theorems~2.4 and~8.2]{Osa12}. It follows from~\citet[Theorem~13.1 and Remark~12.1]{OsaTan20}, together with Theorem~\ref{thm:main}, that this associated labelled solution is strong and is unique in the sense specified there for $\muB$-a.e.~initial configuration. The solution is called {\it infinite-dimensional Dyson Brownian motion with inverse temperature~$\beta$}.

We say that $(\mathcal E_\beta, \mathcal D_\beta)$ is {\it irreducible} if
$F\in\mathcal D_\beta$ and $\mathcal E_\beta(F)=0$ imply that $F$ is constant $\muB$-a.e. It is equivalent to the~$\mathrm{L}^2$-{\it ergodicity} of $\{T^{(\beta)}_t\}_{t\ge0}$:
$$T^{(\beta)}_t f \xrightarrow{t \to \infty} \int_{\mathsf{Conf}(\R)}f(\boldsymbol{\gamma})\,\muB(\dd\boldsymbol{\gamma}) \cquad f \in \mathrm{L}^2(\muB) \comma$$
where the convergence is in $\mathrm{L}^2(\muB)$, 
see~\cite[Proposition 2.3]{ALBEVERIO1997394}. Since $1\in\mathcal D_\beta$, $\mathcal E_\beta(1)=0$, and $\muB$ is a probability measure, the form is recurrent. From the process perspective, ergodicity therefore implies that, for each bounded Borel function $f:\mathsf{Conf}(\R)\to\R$,
$$\frac{1}{t} \int_0^t f(\boldsymbol{\Xi}^{(\beta)}_s) \dd s \xrightarrow{t \to \infty} \int_{\mathsf{Conf}(\R)}f(\boldsymbol{\gamma})\,\muB(\dd\boldsymbol{\gamma}),$$
where the convergence is $\mathbb P^{(\beta)}_{\boldsymbol{\gamma}}$-almost surely for q.e.~$\boldsymbol{\gamma}\in\U(\R)$, see, e.g.,~\citet[Corollary~4.5]{Suz23b} for the proof. 

For the present form,~\citet[Remark~1.6]{assiotis2026collision} applies~\citet[Corollary~I]{Suz23b} to deduce irreducibility from tail triviality and number rigidity.
\begin{cor} \label{c:ER}
The $\mathrm{L}^2$-semigroup~$\bigl(T_t^{(\beta)}\bigr)_{t \ge 0}$ is ergodic for every $\beta>0$.  Equivalently, the Dirichlet form~$(\E_\beta, \mathcal D_\beta)$ is irreducible for every $\beta>0$.
\end{cor}

\paragraph{Collision problem} The~$\mathrm{L}^2$-ergodicity result plays an essential role in providing a complete characterisation of the collision problem of the associated diffusion process in the almost-sure sense. 
The collision set is defined as 
\[
\mathscr C \defi \{\boldsymbol{\gamma}\in\U(\R):\boldsymbol{\gamma}(\{x\})\ge2\text{ for some }x\in\mathbb R\}.
\]
Corollary~\ref{c:ER}, combined with the results of~\citet[Theorem~1.5 and Remark~1.6]{assiotis2026collision}, gives the following collision dichotomy.
\begin{cor} The following dichotomy holds:
\begin{itemize}
    \item If $0<\beta<1$, then $\mathbb P^{(\beta)}_{\boldsymbol{\gamma}}[\tau_{\mathscr{C}}<\infty]=1$ for q.e.~$\boldsymbol{\gamma}$;

\item If $\beta\geq1$, then $\mathbb P^{(\beta)}_{\boldsymbol{\gamma}}[\tau_{\mathscr{C}}=\infty]=1$ for q.e.~$\boldsymbol{\gamma}$.
\end{itemize}
Here $\tau_{\mathscr{C}}$ is the first hitting time of the collision set~$\mathscr{C}$ for the associated diffusion process $(\boldsymbol{\Xi}^{(\beta)}_t)_{t\ge 0}$ starting at the initial deterministic configuration~$\boldsymbol{\gamma}$.
\end{cor}

\subsection{Strategy of proof}\label{sec:strategy}

We give two proofs. The common step in both proofs is to show that every tail event is translation invariant modulo $\muB$ which is the main technical contribution of our paper. For the first proof, if $A\in\cT$ has positive probability, the conditional law $\muB(\,\cdot\mid A)$ is then stationary. Since $A$ is measurable from the exterior of every bounded region, this conditional law also satisfies the canonical DLR equations. Its density with respect to $\muB$ is bounded, so the finite electric-energy condition of \cite{LebleSerfaty2017} is inherited. Then the uniqueness theorem of~\citet[Theorem~1.1]{AssiotisDLRUniqueness} therefore identifies the conditional law with $\muB$, forcing $\muB(A)=1$.

For the second proof, we use the correlation estimates of~\citet[Theorem~1.2]{dumaz2026longrangecorrelationsinebetapoint}. We first prove mixing for products of particle counts in disjoint intervals, and then extend it to all square-integrable functions by approximation. The local exponential moments needed for this approximation follow from~\citet[Theorem~1]{HolcombValko}. Let $\tau_t$ denote translation by $t\in\R$. For a tail event $A$ and $h=\mathbf1_A$, translation invariance and mixing then give
\[
 \muB(A)
 =\mathbb E_{\muB}\bigl[h(h\circ\tau_{-t})\bigr]
 \xrightarrow[t\to+\infty]{}\mathbb E_{\muB}[h]^2
 =\muB(A)^2.
\]
Thus $\muB(A)\in\{0,1\}$.

We stress that translation invariance of tail events does not follow from stationarity alone. For example, let
\[
 \boldsymbol{\gamma}_{\mathsf U}
 \defi\sum_{k\in\mathbb Z}\delta_{k+\mathsf U},
 \qquad \mathsf U\overset{\mathrm d}{=}\mathsf{Unif}[0,1).
\]
This law is stationary, but its random shift is recoverable arbitrarily far away. More precisely, for $n\in\N$, let $u_n(\boldsymbol{\gamma})=x-n$ if $\boldsymbol{\gamma}_{[n,n+1)}=\delta_x$, and set $u_n(\boldsymbol{\gamma})=0$ otherwise. The event
\[
 A_{\mathrm{lat}}
 \defi\{\boldsymbol{\gamma}:\lim_{n\to\infty}u_n(\boldsymbol{\gamma})
       \text{ exists and belongs to }[0,1/2)\}
\]
is a tail event: deleting any finite number of terms does not change membership. On these shifted-lattice configurations it has probability $1/2$, and translation by $1/2$ exchanges it with its complement.

We therefore need a mechanism that can pass from compactly supported shifts, which tail events cannot detect, to global translations. The general idea goes back to~\citet{FrohlichPfister} and~\citet{Georgii} and has been adapted to several settings; see~\citep{Richthammer,leble2024dlr}. Here the mechanism is the entropy estimate developed in \cref{sec:entropy}. For fixed $a\in\R$ and sufficiently large $R$, let $\mathfrak T_{R,\pm}$ be the transformations of configurations induced by
\[
 x\longmapsto x\pm a\phi_R(x),
\]
where the cutoff $\phi_R$ is defined in \eqref{eq:scaledcutoff}. These maps are translations by $\pm a$ on $[-R,R]$ and the identity outside $[-2R,2R]$. The DLR equation permits a conditional change of variables. Adding the two relative entropies cancels the terms linear in the displacement, leaving nonnegative logarithmic remainders. The bounded two-point correlation density from~\citet[Theorem~1, equations~(4)--(5)]{qu2025pair} and~\citet[Theorem~1.8]{AssiotisNajnudel}, together with the scale invariance of $\|\phi_R\|_{\dot{\mathrm H}^{1/2}}$ and the bound $\Lip(\phi_R)=R^{-1}$, yields
\begin{equation}\label{eq:EntropyIneqIntro}
 \sup_{R\in[R_a,\infty)}
 \mathsf{Ent}\left((\mathfrak T_{R,+})_*\muB\middle|\muB\right)<\infty.
\end{equation}

For a tail indicator $h=\mathbf1_A$, put $F_n=\mathbb E_{\muB}[h\mid\cF_{[-n,n]}]$ for $n\in\N$. Then $F_n\to h$ in $\mathrm L^1(\muB)$. Choose $R_n$ large enough that $F_n\circ\mathfrak T_{R_n,+}=F_n\circ\tau_a$, where $\tau_a$ denotes translation by $a$, as defined in \eqref{eq:translationdef}. Since $h\circ\mathfrak T_{R_n,+}=h$, stationarity gives
\[
 \|h\circ\tau_a-h\|_{\mathrm L^1(\muB)}
 \leq\|h-F_n\|_{\mathrm L^1(\muB)}
 +\int_{\mathsf{Conf}(\R)}|F_n-h|\,\dd\bigl((\mathfrak T_{R_n,+})_*\muB\bigr)
 \longrightarrow0.
\]
The last convergence follows because the uniform entropy bound makes the densities of $(\mathfrak T_{R_n,+})_*\muB$ with respect to $\muB$ uniformly integrable; the explicit estimate is given in the proof of Proposition~\ref{prop:tailtranslation}. This proves the tail invariance needed for both concluding arguments in \cref{sec:tail}.

\paragraph{Tool and computational resources disclosure.} 
The initial idea of using an entropy estimate to show translation invariance of tail events arose from  interactive discussions with an AI tool (ChatGPT 5.5) several months ago, although that approach involved lengthy and complicated arguments. The final proofs were developed and  simplified by the authors using the DLR equation by~\cite{DLR} and the uniqueness of the Gibbs measure recently shown by the first author~\cite{AssiotisDLRUniqueness}. 
All mathematical arguments have been carefully verified by the authors, who take full responsibility for any remaining errors.

\subsection{Notation and preliminaries}

We use the intensity-one normalization of $\mathsf{Sine}_\beta$. Any other standard normalization is obtained by a deterministic dilation, and all results quoted below are stated in our convention. Throughout, \(\mathsf{Conf}(\R)\) denotes the ambient configuration space introduced above, while \(\mathsf{Conf}_{\mathrm{s}}(\R)\) denotes its simple subspace. We write configurations as
\[
  \boldsymbol{\gamma}=\sum_{i=1}^N\delta_{x_i},\qquad N\in\N_0\cup\{\infty\},
\]
allowing repeated locations, and sums indexed by $x\in\boldsymbol{\gamma}$ are understood to count multiplicity. For a nonnegative measurable function \(f\), or whenever \(f\) is \(\boldsymbol{\gamma}\)-integrable,
\[
  \angles{f,\boldsymbol{\gamma}}\defi
  \int_{\R}f(x)\boldsymbol{\gamma}(\dd x)=\sum_{i=1}^N f(x_i).
\]
The translation action is defined as
\begin{equation}\label{eq:translationdef}
  \tau_t\boldsymbol{\gamma}\defi\sum_{x\in\boldsymbol{\gamma}}\delta_{x+t},\qquad t\in\R.
\end{equation}
A measurable observable \(F\) is called local if it is measurable with respect to \(\cF_{\mathscr{K}}\) for some compact \(\mathscr{K}\subset\R\).
For probability measures $\nu$ and $\mu$ on a measurable space $(\mathsf{X},\mathscr{X})$, with $\nu$ absolutely continuous with respect to $\mu$, define the relative entropy by
\[
  \mathsf{Ent}(\nu|\mu)\defi
  \int_{\mathsf{X}}
  \log\!\left(\frac{\dd\nu}{\dd\mu}(x)\right)
  \nu(\dd x),
\]
and set the entropy to \(+\infty\) otherwise.
We use the first and second Campbell formulas in the intensity-one normalization:
\begin{align*}
 \mathbb E_{\muB}\left[\sum_{x\in\boldsymbol{\gamma}}f(x)\right]
 &=\int_\R f(x)\dd x,\\
 \mathbb E_{\muB}\left[\sum_{(x,y)\in\boldsymbol{\gamma}_{\neq}^2}g(x,y)\right]
 &=\int_{\R^2}g(x,y)\rho^{(2)}(\dd x,\dd y),
\end{align*}
for nonnegative Borel functions $f$ and $g$. Here $\rho^{(2)}$ is the second factorial correlation measure and $\boldsymbol{\gamma}_{\neq}^2$ denotes ordered pairs of distinct particles. For symmetric $g$, the corresponding unordered pair sum carries a factor $1/2$ in the second formula. Empty products equal one and empty sums equal zero.

\section{Uniform entropy bound}\label{sec:entropy}

The pair correlation representations of~\citet[Theorem~1, equations~(4)--(5)]{qu2025pair} and~\citet[Theorem~1.8]{AssiotisNajnudel} each give a bounded density for the second factorial correlation measure. For Qu--Valk\'o's representation, the cosine expectations have absolute value at most one and the coefficients are absolutely summable, so the bound is uniform in the separation. Both references use intensity $(2\pi)^{-1}$. If $\widetilde r_\beta^{(2)}$ denotes a density in that normalization, the dilation $x\mapsto x/(2\pi)$ gives
\[
 r_\beta^{(2)}(x,y)=(2\pi)^2\widetilde r_\beta^{(2)}(2\pi x,2\pi y).
\]
Thus, we get
\begin{equation}\label{eq:boundedtwopoint}
 \rho^{(2)}(\dd x,\dd y)=r_\beta^{(2)}(x,y)\dd x\dd y,
 \qquad
 M_\beta\defi\|r_\beta^{(2)}\|_{\mathrm{L}^\infty(\R^2)}<\infty.
\end{equation}
For a compactly supported Lipschitz function $\phi$, write
\[
 \|\phi\|_{\dot{\mathrm H}^{1/2}}^2
 \defi
 \iint_{\R^2}\frac{|\phi(x)-\phi(y)|^2}{|x-y|^2}\dd x\dd y.
\]

\begin{lemma}\label{lem:dlrentropy}
Let $\phi:\R\to\R$ be compactly supported and Lipschitz, let $a\in\R$, and assume
\[
 \delta\defi|a|\Lip(\phi)<1.
\]
Put $h_\pm(x)=x\pm a\phi(x)$ and define
\begin{align} \label{d:TP}
 \mathfrak T_{\phi,\pm}\!\left(\sum_{x\in\boldsymbol{\gamma}}\delta_x\right)
 \defi\sum_{x\in\boldsymbol{\gamma}}\delta_{h_\pm(x)}.
\end{align}
Then, we have
\begin{align}
 &\mathsf{Ent}\left((\mathfrak T_{\phi,+})_*\muB\middle|\muB\right)
 +\mathsf{Ent}\left((\mathfrak T_{\phi,-})_*\muB\middle|\muB\right)
 =\int_\R-\log\left(1-a^2|\phi'(x)|^2\right)\dd x \notag\\
 &\qquad\quad+\frac{\beta}{2}\iint_{\R^2}r_\beta^{(2)}(x,y)
 \left[-\log\left(1-a^2\left(\frac{\phi(x)-\phi(y)}{x-y}\right)^2\right)\right]\dd x\dd y.
 \label{eq:dlrentropyidentity}
\end{align}
Both terms on the right are nonnegative and finite. Moreover,
\begin{align}
 &\mathsf{Ent}\left((\mathfrak T_{\phi,+})_*\muB\middle|\muB\right)
 +\mathsf{Ent}\left((\mathfrak T_{\phi,-})_*\muB\middle|\muB\right)\leq\frac{a^2}{1-\delta^2}
 \left(\|\phi'\|_{\mathrm L^2(\R)}^2
 +\frac{\beta M_\beta}{2}\|\phi\|_{\dot{\mathrm H}^{1/2}}^2\right).
 \label{eq:dlrentropybound}
\end{align}
\end{lemma}

\begin{proof}
Fix a Borel representative of $\phi'$ by setting it equal to zero
where $\phi$ is not differentiable. As $\phi$ is Lipschitz, such set is a set of Lebesgue measure zero due to the Rademacher theorem.  Define
\[
 J_{\phi,a}(x)\defi-\log\left(1-a^2|\phi'(x)|^2\right),
 \qquad
 K_{\phi,a}(x,y)\defi-\log\left(
 1-a^2\left(\frac{\phi(x)-\phi(y)}{x-y}\right)^2\right)
\]
for $x\neq y$, and set $K_{\phi,a}(x,x)=0$. Since
$-\log(1-s)\leq s/(1-\delta^2)$ for $0\leq s\leq\delta^2$,
\begin{equation}\label{eq:entropykernelbounds}
 0\leq J_{\phi,a}(x)
 \leq\frac{a^2|\phi'(x)|^2}{1-\delta^2},
 \qquad
 0\leq K_{\phi,a}(x,y)
 \leq\frac{a^2}{1-\delta^2}
       \frac{|\phi(x)-\phi(y)|^2}{|x-y|^2}
 \quad(x\neq y).
\end{equation}
Moreover, $\phi\in\mathrm H^1(\R)$ by the compact support and
Lipschitz continuity properties. Using
\[
 \|\phi(\cdot+z)-\phi\|_{\mathrm L^2(\R)}^2
 \leq\min\left\{
 z^2\|\phi'\|_{\mathrm L^2(\R)}^2,
 4\|\phi\|_{\mathrm L^2(\R)}^2\right\},
\]
and splitting the following integral at $|z|=1$, we obtain
\begin{align} \label{e:EHE}
 \|\phi\|_{\dot{\mathrm H}^{1/2}}^2
 =\int_\R
   \frac{\|\phi(\cdot+z)-\phi\|_{\mathrm L^2(\R)}^2}{z^2}\dd z
 \leq2\|\phi'\|_{\mathrm L^2(\R)}^2
      +8\|\phi\|_{\mathrm L^2(\R)}^2<\infty.
\end{align}
Together with \eqref{eq:boundedtwopoint}, these estimates prove
that both integrals on the right-hand side of
\eqref{eq:dlrentropyidentity} are finite and nonnegative.

We next compute the density of the transformed process~\eqref{d:TP}.
For $x<y$,
\[
 (1-\delta)(y-x)
 \leq h_\pm(y)-h_\pm(x)
 \leq(1+\delta)(y-x).
\]
Thus $h_\pm$ are increasing bi-Lipschitz homeomorphisms, with
$h_\pm'=1\pm a\phi'>0$ almost everywhere. Choose
$\Lambda=(-b,b)$ such that $\supp\phi\Subset\Lambda$.
Then $h_\pm(\Lambda)=\Lambda$, and $h_\pm$ fixes $\Lambda^c$.
In particular, the original and transformed processes have
the same exterior configuration.

Let $\pi_\Lambda$ denote the law of the exterior
configuration under $\muB$, i.e., $\pi_\Lambda$ is the push-forward of $\muB$ by the restriction map~$\boldsymbol{\gamma} \mapsto \boldsymbol{\gamma}|_{\Lambda^c}$. By
\citep[Theorem~1.1(B),(C)]{DLR}, for $\pi_\Lambda$-almost every
exterior configuration $\boldsymbol{\eta}$, the number
$N=N_\Lambda(\boldsymbol{\eta})$ of interior particles is fixed.
Furthermore, the law disintegrated at the exterior configuration~$\boldsymbol{\eta}$ has the density
\begin{equation}\label{eq:orderedconditionaldensity}
 p_{\boldsymbol{\eta}}(\mathbf x)
 =\frac{N!}{Z_\Lambda(\boldsymbol{\eta})}
   \prod_{1\leq i<j\leq N}|x_i-x_j|^\beta
   \prod_{i=1}^N\omega(x_i\mid\boldsymbol{\eta})
\end{equation}
with respect to Lebesgue measure on
\[
 \Lambda^N_<\defi
 \{\mathbf x\in\Lambda^N:x_1<\cdots<x_N\}.
\]
The factor $N!$ accounts for ordering the particles, and
\begin{align} \label{e:LDM}
 \omega(x\mid\boldsymbol{\eta})
 =\lim_{m\to\infty}\omega_m(x\mid\boldsymbol{\eta}),
 \qquad
 \omega_m(x\mid\boldsymbol{\eta})
 \defi\prod_{\substack{u\in\boldsymbol{\eta}\\|u|\leq m}}
       \left|1-\frac xu\right|^\beta.
\end{align}
The limiting functions
are positive and finite almost everywhere.

For a fixed exterior configuration, put
\[
 H_\pm\mathbf x\defi
 (h_\pm(x_1),\ldots,h_\pm(x_N)),
 \qquad
 j_\pm(\mathbf x)\defi
 \prod_{i=1}^N(1\pm a\phi'(x_i)).
\]
Then $H_\pm$ is a bi-Lipschitz bijection of $\Lambda^N_<$,
and its Jacobian is $j_\pm>0$ almost everywhere.
Let $q_{\boldsymbol{\eta},\pm}$ denote the density of the law of the
transformed interior tuple~$H_\pm\mathbf x$, i.e., 
for every bounded Borel
function $f$ on $\Lambda^N_<$,
\[
 \int_{\Lambda^N_<}f(\mathbf y)
       q_{\boldsymbol{\eta},\pm}(\mathbf y)\dd\mathbf y
 =\int_{\Lambda^N_<}f(H_\pm\mathbf x)
       p_{\boldsymbol{\eta}}(\mathbf x)\dd\mathbf x.
\]
The substitution $\mathbf y=H_\pm\mathbf x$ gives
\begin{equation}\label{eq:conditionalchangeofvariables}
 q_{\boldsymbol{\eta},\pm}(\mathbf y)
 =\frac{p_{\boldsymbol{\eta}}(H_\pm^{-1}\mathbf y)}
        {j_\pm(H_\pm^{-1}\mathbf y)},
 \qquad
 q_{\boldsymbol{\eta},\pm}(H_\pm\mathbf x)
 =\frac{p_{\boldsymbol{\eta}}(\mathbf x)}
        {j_\pm(\mathbf x)}
\end{equation}
almost everywhere.

Write $Q_\pm\defi(\mathfrak T_{\phi,\pm})_*\muB$.
Since the transformation leaves the exterior unchanged,
disintegration gives, for every bounded Borel function $F$
on $\mathsf{Conf}(\R)$,
\[
 \int F\dd Q_\pm
 =\int\left[
   \int_{\Lambda^{N_\Lambda(\boldsymbol{\eta})}_<}
   F\left(\boldsymbol{\eta}
          +\sum_{i=1}^{N_\Lambda(\boldsymbol{\eta})}\delta_{y_i}\right)
   q_{\boldsymbol{\eta},\pm}(\mathbf y)\dd\mathbf y
   \right]\pi_\Lambda(\dd\boldsymbol{\eta}).
\]
The same formula for $\muB$ in plaec of $Q_+$ in the LHS gives  $p_{\boldsymbol{\eta}}$ in place
of $q_{\boldsymbol{\eta},\pm}$ in the RHS. Both conditional densities are
positive almost everywhere. Hence $Q_\pm\sim\muB$, and their
Radon--Nikodym derivatives satisfy
\[
 R_\pm\left(\boldsymbol{\eta}
                 +\sum_{i=1}^N\delta_{y_i}\right)
 =\frac{q_{\boldsymbol{\eta},\pm}(\mathbf y)}
        {p_{\boldsymbol{\eta}}(\mathbf y)},
 \qquad
 R_\pm\defi\frac{\dd Q_\pm}{\dd\muB}.
\]
In particular, by \eqref{eq:conditionalchangeofvariables},
\[
 R_\pm(\mathfrak T_{\phi,\pm}\boldsymbol{\gamma})
 =\frac{p_{\boldsymbol{\eta}}(\mathbf x)}
        {p_{\boldsymbol{\eta}}(H_\pm\mathbf x)j_\pm(\mathbf x)},
 \qquad
 \boldsymbol{\gamma}
 =\boldsymbol{\eta}+\sum_{i=1}^N\delta_{x_i}.
\]

Define
\[
 L_\pm(\boldsymbol{\gamma})
 \defi\log R_\pm(\mathfrak T_{\phi,\pm}\boldsymbol{\gamma}).
\]
Substituting \eqref{eq:orderedconditionaldensity}, we have
\begin{align} \label{e:ITJ}
 L_\pm(\boldsymbol{\gamma})
 ={}&-\sum_{i=1}^N\log h_\pm'(x_i)
 +\beta\sum_{1\leq i<j\leq N}
   \log\frac{|x_i-x_j|}
            {|h_\pm(x_i)-h_\pm(x_j)|} +\sum_{i=1}^N
   \log\frac{\omega(x_i\mid\boldsymbol{\eta})}
             {\omega(h_\pm(x_i)\mid\boldsymbol{\eta})}.
\end{align}
We add the two signs separately in these three terms.
For the Jacobians and the interior interactions in~\eqref{e:ITJ},
\[
 -\log h_+'(x)-\log h_-'(x)=J_{\phi,a}(x),
\]
and, for $x\neq y$,
\[
 \log\frac{|x-y|}{|h_+(x)-h_+(y)|}
 +\log\frac{|x-y|}{|h_-(x)-h_-(y)|}
 =K_{\phi,a}(x,y).
\]
For the exterior interactions in~\eqref{e:ITJ}, first use the same finite cutoff
$m$ for $\omega_n$ in~\eqref{e:LDM}. Since $\phi(u)=0$ for
$u\in\boldsymbol{\eta}$, direct multiplication gives
\[
 \log\frac{\omega_m(x\mid\boldsymbol{\eta})}
              {\omega_m(h_+(x)\mid\boldsymbol{\eta})}
 +\log\frac{\omega_m(x\mid\boldsymbol{\eta})}
              {\omega_m(h_-(x)\mid\boldsymbol{\eta})}
 =\beta\sum_{\substack{u\in\boldsymbol{\eta}\\|u|\leq m}}
       K_{\phi,a}(x,u).
\]
By Fubini's theorem, for almost every exterior configuration
the exterior products converge to positive finite limits for
Lebesgue-almost every $x\in\Lambda$. Since $h_\pm$ and their
inverses preserve Lebesgue-null sets, these limits exist
simultaneously at $x_i$ and $h_\pm(x_i)$ for almost every tuple. Thus the logarithms on the left
converge. The sums on the right increase with $m$, because
$K_{\phi,a}\geq0$. Passing to $m \to \infty$, it follows that
\[
 \log\frac{\omega(x\mid\boldsymbol{\eta})}
              {\omega(h_+(x)\mid\boldsymbol{\eta})}
 +\log\frac{\omega(x\mid\boldsymbol{\eta})}
              {\omega(h_-(x)\mid\boldsymbol{\eta})}
 =\beta\sum_{u\in\boldsymbol{\eta}}K_{\phi,a}(x,u)
\]
at the interior particle positions almost surely.

Combining the three contributions in~\eqref{e:ITJ} yields
\begin{equation}\label{eq:pairedconfigurationentropy}
 L_+(\boldsymbol{\gamma})+L_-(\boldsymbol{\gamma})
 =\sum_{x\in\boldsymbol{\gamma}}J_{\phi,a}(x)
  +\frac{\beta}{2}
   \sum_{\substack{x,y\in\boldsymbol{\gamma}\\x\neq y}}
       K_{\phi,a}(x,y)
 \defi S(\boldsymbol{\gamma}).
\end{equation}
Here the double sum counts ordered pairs, and
$J_{\phi,a}$ vanishes on $\Lambda^c$, and $K_{\phi,a}(x,y)=0$
when both $x$ and $y$ belong to $\Lambda^c$.

By unit intensity, the Campbell formulas, and Tonelli's theorem, and the estimates~\eqref{eq:entropykernelbounds} and~\eqref{e:EHE}, 
\begin{align*}
 \mathbb E_{\muB}[S]
 &=\int_\R J_{\phi,a}(x)\dd x
   +\frac{\beta}{2}\iint_{\R^2}
       r_\beta^{(2)}(x,y)K_{\phi,a}(x,y)\dd x\dd y\\
 &\leq\frac{a^2}{1-\delta^2}
 \left(\|\phi'\|_{\mathrm L^2(\R)}^2
       +\frac{\beta M_\beta}{2}
        \|\phi\|_{\dot{\mathrm H}^{1/2}}^2\right)
 <\infty.
\end{align*}

It remains to identify the expectations of $L_\pm$ with the
individual entropies. For $v\in\R$, write
$v^-\defi\max\{-v,0\}$. By the definition of $Q_\pm$,
\[
 \mathbb E_{\muB}[(L_\pm)^-]
 =\int(\log R_\pm)^-\dd Q_\pm
 =\int R_\pm(\log R_\pm)^-\dd\muB
 \leq\mathrm e^{-1},
\]
where we used $r(\log r)^-\leq\mathrm e^{-1}$ for $r>0$.
Moreover, \eqref{eq:pairedconfigurationentropy} implies
\[
 (L_\pm)^+\leq S+(L_\mp)^-.
\]
Thus both $L_+$ and $L_-$ are integrable, and
\[
 \mathbb E_{\muB}[L_\pm]
 =\int\log R_\pm\dd Q_\pm
 =\mathsf{Ent}(Q_\pm\mid\muB).
\]
Taking expectations in \eqref{eq:pairedconfigurationentropy}
now proves \eqref{eq:dlrentropyidentity}. The preceding estimate
for $\mathbb E_{\muB}[S]$ proves~\eqref{eq:dlrentropybound}.
\end{proof}

We now choose compactly supported deformations that agree with translations on expanding intervals. Define the even Lipschitz function
\[
 \chi(x)\defi
 \begin{cases}
  1,&|x|\leq1,\\
  2-|x|,&1<|x|<2,\\
  0,&|x|\geq2,
 \end{cases}
\]
and, for $R\geq1$, set
\begin{equation}\label{eq:scaledcutoff}
 \phi_R(x)\defi\chi(x/R).
\end{equation}

\begin{lemma}\label{lem:scaledcutoff}
For every $R\geq1$,
\begin{align}
 \|\phi_R\|_{\dot{\mathrm H}^{1/2}}^2
 &=\|\chi\|_{\dot{\mathrm H}^{1/2}}^2,
  \label{eq:scaledhhalf}\\*
 \int_\R|\phi_R'(x)|^2\dd x
 &=\frac2R,
  \label{eq:scaledderivative}\\*
 \Lip(\phi_R)&=\frac1R.
 \label{eq:scaledlip}
\end{align}
\end{lemma}

\begin{proof}
The first identity follows from the change of variables $x=Ru$, $y=Rv$. The weak derivative has absolute value $R^{-1}$ on $\{R<|x|<2R\}$ and vanishes elsewhere, which proves the remaining identities.
\end{proof}

Let $\mathfrak T_{R,\pm}$ be the transformation of $\mathsf{Conf}(\R)$ induced by
\begin{equation}\label{eq:expandingtranslation}
 \mathsf T_{R,\pm}(x)\defi x\pm a\phi_R(x).
\end{equation}
For all sufficiently large $R$, these maps are increasing bi-Lipschitz homeomorphisms of $\R$, equal to the identity outside $[-2R,2R]$ and to translation by $\pm a$ on $[-R,R]$.

\begin{proposition}\label{prop:uniformentropy}
For every $a\in\R$, there are constants $R_a<\infty$ and $C_{\beta,a}<\infty$ such that
\begin{equation}\label{eq:uniformentropy}
 \sup_{R\in[R_a,\infty)}
 \left\{
 \mathsf{Ent}\left((\mathfrak T_{R,+})_*\muB\middle|\muB\right)
 +\mathsf{Ent}\left((\mathfrak T_{R,-})_*\muB\middle|\muB\right)
 \right\}
 \leq C_{\beta,a}.
\end{equation}
\end{proposition}

\begin{proof}
Take $R_a=\max\{1,2|a|\}$. For $R\geq R_a$, Lemma~\ref{lem:scaledcutoff} gives
\[
 |a|\Lip(\phi_R)=\frac{|a|}{R}\leq\frac12.
\]
Hence Lemma~\ref{lem:dlrentropy} yields
\begin{align*}
&\mathsf{Ent}\left((\mathfrak T_{R,+})_*\muB\middle|\muB\right)
 +\mathsf{Ent}\left((\mathfrak T_{R,-})_*\muB\middle|\muB\right)\leq\frac{a^2}{1-a^2/R^2}
 \left(\frac2R+\frac{\beta M_\beta}{2}\|\chi\|_{\dot{\mathrm H}^{1/2}}^2\right),
\end{align*}
which is uniformly bounded for $R\geq R_a$.
\end{proof}

\section{Tail invariance and proofs of tail triviality}\label{sec:tail}

Recall that $\muB$ is concentrated on $\mathsf{Conf}_{\mathrm{s}}(\R)$. The \(\sigma\)-fields $\cF_D$ and $\cT$ retain their definitions on the ambient space $\mathsf{Conf}(\R)$; on $\mathsf{Conf}_{\mathrm{s}}(\R)$, the same notation denotes their traces.

The following elementary observation identifies the exact invariance implied directly by the tail \(\sigma\)-field: compactly supported homeomorphisms do not change a tail event.

Let $A \in\cT$, and let $s:\R\to\R$ be a homeomorphism that is the identity outside a compact set. Define its induced action on configurations by
\[
 \mathfrak{S}_s\!\left(\sum_{x\in\boldsymbol{\gamma}}\delta_x\right)
 \defi\sum_{x\in\boldsymbol{\gamma}}\delta_{s(x)}
 =s_*\boldsymbol{\gamma}.
\]
The induced map $\mathfrak S_s$ is continuous in the vague topology, since for $f\in C_c(\R)$ one has $\langle f,s_*\boldsymbol{\gamma}\rangle=\langle f\circ s,\boldsymbol{\gamma}\rangle$ and $f\circ s\in C_c(\R)$.
\begin{lemma}\label{lem:tailinvariance}
\begin{equation}
 \mathbf{1}_{A}\circ\mathfrak{S}_s
 =\mathbf{1}_{A}.
\end{equation}
\end{lemma}

\begin{proof}
Choose $M>0$ such that $s$ is the identity outside $[-M,M]$. Injectivity implies that no point of $[-M,M]$ is mapped into its complement. Thus a configuration and its image under $\mathfrak S_s$ have the same restriction to $[-M,M]^c$. Since $A\in\cF_{[-M,M]^c}$, membership in $A$ is unchanged.
\end{proof}

Proposition \ref{prop:uniformentropy} now allows the compactly supported invariance from Lemma \ref{lem:tailinvariance} to pass to global translations. The resulting translation invariance gives stationarity of tail-conditioned laws in the first proof and allows mixing to be applied to tail indicators in the second proof.

\begin{proposition}\label{prop:tailtranslation}
Every $A \in\cT$ is translation invariant modulo $\muB$: for every $a\in\R$,
\begin{equation}\label{eq:tailtranslation}
 \mathbf{1}_{A}\circ\tau_a
 =\mathbf{1}_{A}
 \qquad\muB\text{-a.e.}
\end{equation}
\end{proposition}

\begin{proof}
Fix $a\in\R$ and write $h=\mathbf{1}_{A}$. For $n\in\N$, define
\[
 F_n\defi\mathbb{E}_{\muB}[h\mid\cF_{[-n,n]}].
\]
Let
\[
 r_n:\mathsf{Conf}(\R)\longrightarrow\mathsf{Conf}([-n,n]),
 \qquad
 r_n(\boldsymbol{\gamma})\defi\boldsymbol{\gamma}_{[-n,n]},\qquad n\in\N.
\]
The two configuration spaces are standard Borel, and $\cF_{[-n,n]}=\sigma(r_n)$. Thus, by the Doob--Dynkin lemma, we may choose a version of $F_n$ and a Borel function $f_n:\mathsf{Conf}([-n,n])\to[0,1]$ such that
\[
 F_n=f_n\circ r_n.
\]

Moreover,
\[
 \sigma\!\left(\bigcup_{n\in\N}\cF_{[-n,n]}\right)
 =\mathscr{B}(\mathsf{Conf}(\R)).
\]
Hence the martingale convergence theorem gives
\begin{equation}\label{eq:localtailapproximation}
 0\leq F_n\leq1,
 \qquad
 \|F_n-h\|_{\mathrm{L}^1(\muB)}\longrightarrow0.
\end{equation}

Choose $R_n\to\infty$ so that $R_n>n+|a|+1$ and $R_n\geq R_a$ for every $n\in\N$. The restrictions of $\mathfrak{T}_{R_n,+}\boldsymbol{\gamma}$ and $\tau_a\boldsymbol{\gamma}$ to $[-n,n]$ coincide for every configuration $\boldsymbol{\gamma}$. Indeed, a point whose translate lies in $[-n,n]$ starts in $[-n-|a|,n+|a|]\subset[-R_n,R_n]$, where $\mathsf{T}_{R_n,+}(x)=x+a$; a point starting outside $[-R_n,R_n]$ cannot be moved into $[-n,n]$, since
\[
 |\mathsf{T}_{R_n,+}(x)|\geq|x|-|a|>n.
\]
Consequently,
\begin{equation}\label{eq:localtranslationagreement}
 F_n\circ\mathfrak{T}_{R_n,+}=F_n\circ\tau_a.
\end{equation}
On the other hand, Lemma \ref{lem:tailinvariance} gives
\[
 h\circ\mathfrak{T}_{R_n,+}=h.
\]

Set $\nu_n=(\mathfrak{T}_{R_n,+})_*\muB$ for $n\in\N$. Stationarity of $\muB$, \eqref{eq:localtranslationagreement}, and the preceding tail identity give
\begin{align*}
 \|h\circ\tau_a-h\|_{\mathrm{L}^1(\muB)}
 &\leq
 \|(h-F_n)\circ\tau_a\|_{\mathrm{L}^1(\muB)}
 +\|(F_n-h)\circ\mathfrak{T}_{R_n,+}\|_{\mathrm{L}^1(\muB)}\\
 &=\|h-F_n\|_{\mathrm{L}^1(\muB)}
 +\int_{\mathsf{Conf}(\R)}
 |F_n(\boldsymbol{\gamma})-h(\boldsymbol{\gamma})|\nu_n(\dd\boldsymbol{\gamma}).
\end{align*}
The first term tends to zero by \eqref{eq:localtailapproximation}. We verify the convergence of the second term using the standard entropy criterion for uniform integrability; see, e.g., ~\citet[Lemma~4.11(2), pp.~143--144]{Eberle2016StochasticAnalysis}.  Proposition~\ref{prop:uniformentropy} gives
\[
 C\defi\sup_{n\in\N}\mathsf{Ent}(\nu_n\mid\muB)<\infty.
\]
Set $\rho_n=\dd\nu_n/\dd\muB$, and write $\log^+u=\max\{\log u,0\}$ for $u>0$, with $0\log^+0=0$. The negative part of $u\log u$ is at most $\mathrm e^{-1}$, so 
\[
 \sup_{n\in\N}\int_{\mathsf{Conf}(\R)}\rho_n\log^+\rho_n\dd\muB\leq C+\mathrm e^{-1}.
\]
By, e.g., ~\citet[Lemma~4.11(2), pp.~143--144]{Eberle2016StochasticAnalysis}, the densities $\rho_n$ are uniformly integrable:
\begin{equation}\label{eq:entropydensitytails}
 \lim_{M\to+\infty}\sup_{n\in\N}
 \int_{\{\boldsymbol{\gamma}\in\mathsf{Conf}(\R):\rho_n(\boldsymbol{\gamma})>M\}}\rho_n(\boldsymbol{\gamma})\,\muB(\dd\boldsymbol{\gamma})=0.
\end{equation}
Put $G_n=|F_n-h|\in[0,1]$. Splitting according to whether $\rho_n\leq M$ gives, for every real $M>1$,
\begin{align*}
 \int_{\mathsf{Conf}(\R)}G_n\dd\nu_n
 &=\int_{\mathsf{Conf}(\R)}G_n\rho_n\dd\muB\\
 &\leq M\int_{\mathsf{Conf}(\R)}G_n\dd\muB+\sup_{k\in\N}\int_{\{\boldsymbol{\gamma}\in\mathsf{Conf}(\R):\rho_k(\boldsymbol{\gamma})>M\}}\rho_k(\boldsymbol{\gamma})\,\muB(\dd\boldsymbol{\gamma}).
\end{align*}
For fixed $M$, the first term tends to zero by \eqref{eq:localtailapproximation}. Letting $M\to\infty$ then makes the second term tend to zero by \eqref{eq:entropydensitytails}. This proves \eqref{eq:tailtranslation}.
\end{proof}

We now give two proofs of \cref{thm:main}, both using Proposition~\ref{prop:tailtranslation}.

\paragraph{DLR uniqueness input.}
We recall the form of~\citet[Theorem~1.1]{AssiotisDLRUniqueness} that we need. Let $K^\beta_\Lambda$ denote the canonical resampling operator introduced in~\citet[Section~2.1]{DLR}, in the formulation of~\citet[Definitions~2.1--2.2]{AssiotisDLRUniqueness}: it replaces the configuration in a bounded Borel region $\Lambda$ of positive Lebesgue measure according to the canonical $\beta$-DLR kernel, retaining the exterior and the number of interior particles. A stationary Borel probability measure $P$ satisfying these canonical DLR equations and $\mathcal W^{\mathrm{div}}(P)<\infty$ equals $\muB$. Here $\mathcal W^{\mathrm{div}}$ is the divergence-compatible renormalised electric energy of Lebl\'{e}-Serfaty \cite{LebleSerfaty2017} recalled in the form we need in ~\citet[Definition~2.6]{AssiotisDLRUniqueness}. Then, in the notation of that definition and the lower bound stated immediately after it,
\begin{equation}\label{eq:electricenergyinput}
 w\defi\widetilde{\mathcal W}^{\mathrm{div}}\geq-C_{\mathrm{el}},
 \qquad
 \mathcal W^{\mathrm{div}}(P)=\int_{\mathsf{Conf}(\R)}w(\boldsymbol{\gamma})\,P(\dd\boldsymbol{\gamma}),
\end{equation}
where $w$ is a function of the configuration and $C_{\mathrm{el}}<\infty$ is universal.

Moreover, \mbox{$\mathcal W^{\mathrm{div}}(\muB)<\infty$}. This finite-energy assertion follows from~\citet[Corollary~1.2 and Sections~2.7.3--2.7.4]{LebleSerfaty2017}, as recorded in~\citet[Lemma~5.2]{AssiotisDLRUniqueness}. The canonical DLR equations were proved in~\citet[Theorems~1.1 and~2.1]{DLR}; their formulation used here is recorded in~\citet[Lemma~2.9]{AssiotisDLRUniqueness}.

\begin{proof}[First proof of \cref{thm:main}: DLR uniqueness]
Let $A\in\cT$. If $\muB(A)=0$, there is nothing to prove. Otherwise put
\[
 p\defi\muB(A)>0,\qquad h\defi\mathbf1_A,\qquad
 \nu_A(B)\defi\frac{\muB(B\cap A)}p,\quad B\in\mathscr B(\mathsf{Conf}(\R)).
\]
Thus $\nu_A$ is a Borel probability measure and
\[
 \frac{\dd\nu_A}{\dd\muB}=\frac hp\leq\frac1p.
\]
By Proposition~\ref{prop:tailtranslation}, $h\circ\tau_t=h$ $\muB$-a.e. for each $t\in\R$. Stationarity of $\muB$ therefore gives, for every bounded Borel $F:\mathsf{Conf}(\R)\to\R$,
\begin{align*}
 \nu_A(F\circ\tau_t)
 &=\frac1p\int_{\mathsf{Conf}(\R)}h(\boldsymbol{\gamma})F(\tau_t\boldsymbol{\gamma})\,\muB(\dd\boldsymbol{\gamma})\\*
 &=\frac1p\int_{\mathsf{Conf}(\R)}h(\tau_{-t}\boldsymbol{\gamma})F(\boldsymbol{\gamma})\,\muB(\dd\boldsymbol{\gamma})
 =\nu_A(F).
\end{align*}
Hence $\nu_A$ is stationary.

By the number rigidity and the DLR theorem, $K^\beta_\Lambda F=\Gamma^\beta_\Lambda F$ $\muB$-a.e. 
This equality holds also $\nu_A$ a.e.~by absolute continuity.
For every bounded Borel $\Lambda$ of positive Lebesgue measure, $A\in\cF_{\Lambda^c}$, since $A\in\cF_{[-R,R]^c}\subset\cF_{\Lambda^c}$ whenever $\Lambda\subset[-R,R]$. Canonical resampling preserves the exterior, so
\[
 K^\beta_\Lambda(hF)=hK^\beta_\Lambda F
 \qquad\muB\text{-a.e.}
\]
The canonical DLR identity for $\muB$, applied to $hF$, consequently gives
\[
 \nu_A(K^\beta_\Lambda F)
 =\frac1p\muB(hK^\beta_\Lambda F)
 =\frac1p\muB(K^\beta_\Lambda(hF))
 =\frac1p\muB(hF)
 =\nu_A(F).
\]
Thus $\nu_A$ satisfies the canonical DLR equations.

Finally, the nonnegative shift in \eqref{eq:electricenergyinput} and the bounded density imply
\begin{align*}
 \mathcal W^{\mathrm{div}}(\nu_A)+C_{\mathrm{el}}
 &=\frac1p\int_A\bigl(w(\boldsymbol{\gamma})+C_{\mathrm{el}}\bigr)\,\muB(\dd\boldsymbol{\gamma})\leq\frac{\mathcal W^{\mathrm{div}}(\muB)+C_{\mathrm{el}}}{p}<\infty.
\end{align*}
All hypotheses of~\citet[Theorem~1.1]{AssiotisDLRUniqueness} therefore hold, and $\nu_A=\muB$. It follows that $1=\nu_A(A)=\muB(A)$, completing the proof.
\end{proof}

\begin{proof}[Second proof of \cref{thm:main}: correlation decay]
For a bounded Borel set $B\subset\R$, define 
\[
 \mathsf N_B(\boldsymbol{\gamma})\defi\boldsymbol{\gamma}(B),
 \qquad\boldsymbol{\gamma}\in\mathsf{Conf}(\R).
\]
Note that, this differs from the notation $N_\Lambda(\boldsymbol{\eta})$ for the particle number determined by an exterior configuration due to the number rigidity. Let
\begin{gather*}
 \mathcal I\defi\{[a,b):a,b\in\mathbb Q,\ a<b\},
\\
 \mathcal V\defi\Span_{\R}\left(
 \{1\}\cup\left\{\prod_{j=1}^p\mathsf N_{I_j}:
 \begin{array}{l}
 p\in\N,\ I_1,\ldots,I_p\in\mathcal I,\\
 I_1,\ldots,I_p\text{ pairwise disjoint}
 \end{array}\right\}\right),
\end{gather*}
where $\Span_{\R}$ denotes the real linear span. The overcrowding estimate of~\citet[Theorem~1]{HolcombValko}, after dilation to our intensity-one normalization, gives $\mathbb E_{\muB}[4^{\mathsf N_K}]<\infty$ for every bounded Borel set $K\subset\R$. In particular, $\mathcal V\subset\mathrm L^2(\muB)$.

We first show that $\mathcal V$ is dense in $\mathrm L^2(\muB)$. Let $K$ be a finite union of intervals in $\mathcal I$. For each $n\in\N$, choose a finite partition $\mathcal P_n(K)$ of $K$ into intervals in $\mathcal I$, each of length at most $1/n$. For $K=\varnothing$, use the empty partition. Define
\begin{align*}
 V_K(\boldsymbol{\gamma})&\defi\mathbf1_{\{\mathsf N_K(\boldsymbol{\gamma})=0\}},\\
 V_{K,n}(\boldsymbol{\gamma})&\defi
 \prod_{I\in\mathcal P_n(K)}\bigl(1-\mathsf N_I(\boldsymbol{\gamma})\bigr),
 \qquad\boldsymbol{\gamma}\in\mathsf{Conf}(\R).
\end{align*}
Expanding the finite product shows that $V_{K,n}\in\mathcal V$, and
\[
 |V_{K,n}(\boldsymbol{\gamma})|^2\leq4^{\mathsf N_K(\boldsymbol{\gamma})},
 \qquad\boldsymbol{\gamma}\in\mathsf{Conf}(\R),\quad n\in\N.
\]
On simple locally finite configurations, every cell eventually contains at most one particle. Thus $V_{K,n}\to V_K$ as $n\to\infty$, $\muB$-almost surely and, by dominated convergence, in $\mathrm L^2(\muB)$. The void events $\{\boldsymbol{\gamma}:\mathsf N_K(\boldsymbol{\gamma})=0\}$, with $K$ ranging over finite unions of intervals in $\mathcal I$, including $\varnothing$, form a $\pi$-system. Their traces on $\mathsf{Conf}_{\mathrm s}(\R)$ generate its Borel $\sigma$-field: for $I\in\mathcal I$ and every simple configuration $\boldsymbol{\gamma}$,
\[
 \mathsf N_I(\boldsymbol{\gamma})
 =\lim_{n\to\infty}\sum_{J\in\mathcal P_n(I)}
 \bigl(1-V_J(\boldsymbol{\gamma})\bigr).
\]
Since $\muB$ is concentrated on simple configurations, the monotone-class theorem makes the linear span of these $V_K$ dense in $\mathrm L^2(\muB)$, and hence so is $\mathcal V$.

Let $p,q\in\N$, let $I_1,\ldots,I_p\in\mathcal I$ be pairwise disjoint, and let $J_1,\ldots,J_q\in\mathcal I$ be pairwise disjoint. Set
\[
 P\defi\prod_{i=1}^p\mathsf N_{I_i},\qquad
 Q\defi\prod_{j=1}^q\mathsf N_{J_j},\qquad
 J+t\defi\{x+t:x\in J\}\quad(J\subset\R,\ t\in\R).
\]
The intervals $I_1,\ldots,I_p,J_1+t,\ldots,J_q+t$ are pairwise disjoint for all sufficiently large real $t$. Since $\mathsf N_J\circ\tau_{-t}=\mathsf N_{J+t}$, the correlation estimate of~\citet[Theorem~1.2]{dumaz2026longrangecorrelationsinebetapoint} gives
\begin{align*}
 \mathbb E_{\muB}\bigl[P(Q\circ\tau_{-t})\bigr]
 &=\mathbb E_{\muB}\left[
 \left(\prod_{i=1}^p\mathsf N_{I_i}\right)
 \left(\prod_{j=1}^q\mathsf N_{J_j+t}\right)\right]\\
 &\longrightarrow\mathbb E_{\muB}[P]\mathbb E_{\muB}[Q]
 \qquad(t\to+\infty).
\end{align*}
Choose $R>0$ such that $I_i,J_j\subset(-R,R)$ for every $1\leq i\leq p$ and $1\leq j\leq q$. After translating both families by $-R$ and multiplying coordinates by $2\pi$, we apply that theorem to the fixed intervals $2\pi(I_i-R)$ and $2\pi(J_j+R)$, with separation $2\pi(t-2R)\to\infty$. The limit is immediate when either factor is constant; linearity therefore gives it for every $P,Q\in\mathcal V$.

For $F,G\in\mathrm L^2(\muB)$ and $P,Q\in\mathcal V$, stationarity and Cauchy--Schwarz now give
\begin{align*}
 &\limsup_{t\to+\infty}
 \left|\mathbb E_{\muB}\bigl[F(G\circ\tau_{-t})\bigr]
       -\mathbb E_{\muB}[F]\mathbb E_{\muB}[G]\right|\\
 &\qquad\leq
 \|F-P\|_{\mathrm L^2(\muB)}\|G\|_{\mathrm L^2(\muB)}
 +\|P\|_{\mathrm L^2(\muB)}\|G-Q\|_{\mathrm L^2(\muB)}.
\end{align*}
Density makes the right-hand side arbitrarily small, proving mixing as $t\to+\infty$; stationarity gives the limit as $t\to-\infty$ as well.
Finally, for $A\in\cT$ and $h=\mathbf1_A$, Proposition~\ref{prop:tailtranslation} gives $h\circ\tau_{-t}=h$ $\muB$-a.e. for each $t\in\R$. Therefore
\[
 \muB(A)=\mathbb E_{\muB}\bigl[h(h\circ\tau_{-t})\bigr]
 \xrightarrow[t\to+\infty]{}\mathbb E_{\muB}[h]^2=\muB(A)^2,
\]
so $\muB(A)\in\{0,1\}$.
\end{proof}

\begingroup
\interlinepenalty=10000
\setlength{\bibsep}{5pt plus 1pt minus 1pt}
\bibliographystyle{plainnat}
\bibliography{References_final_version}
\endgroup

\bigskip 

\noindent{\sc School of Mathematics, University of Edinburgh, James Clerk Maxwell Building, Peter Guthrie Tait Rd, Edinburgh EH9 3FD, U.K.}\newline
\href{mailto:theo.assiotis@ed.ac.uk}{\small theo.assiotis@ed.ac.uk}

\bigskip

\noindent{\sc Department of Mathematical Science, Durham University, DH1 3LE, Durham, United Kingdom}\newline
\href{mailto:kohei.suzuki@durham.ac.uk}{\small kohei.suzuki@durham.ac.uk}

\end{document}